\documentclass[a4paper,11pt]{amsart}
\usepackage[english]{babel}
\usepackage[centertags]{amsmath}
\usepackage{latexsym, amsfonts, amssymb, graphicx, mathrsfs}
\usepackage{hyperref}
\usepackage{fourier}

\usepackage[all,dvips]{xy}
\usepackage{graphicx}
\usepackage[latin1]{inputenc}
\usepackage{url}
\usepackage{enumerate}
\usepackage{multirow}
\usepackage{fourier}

\theoremstyle{plain}
\newtheorem{thm}{Theorem}[section]
\newtheorem{lem}[thm]{Lemma}
\newtheorem{prop}[thm]{Proposition}
\newtheorem{cor}[thm]{Corollary}
\newtheorem{ass}{Assumption}
\newtheorem{question}{Question}

\theoremstyle{definition}
\newtheorem{defn}[thm]{Definition}
\newtheorem{nota}[thm]{Notation}

\theoremstyle{remark}
\newtheorem{rem}[thm]{Remark}

\newtheorem{exmp}[thm]{Example}

\def\A{{\mathbf A}}

\def\Zc{{\mathscr Z}}
\def\P{{\mathbf P}}

\def\L{{\mathscr L}}

\def\H{\mathscr{H}}
\def\X{\mathscr{X}}
\def\Y{\mathscr{Y}}
\def\V{\mathcal{V}}
\def\E{\mathcal{E}}
\def\G{\mathcal{G}}
\newcommand{\F}{\mathbf{F}}

\newcommand{\eqdef}{\stackrel{\text{def}}{=}}
\renewcommand{\leq}{\leqslant} 
\renewcommand{\geq}{\geqslant} 
\newcommand{\red}{\textrm{red}}

\newcommand{\fract}[2]{\hbox{\leavevmode
\kern.1em \raise .5ex \hbox{\the\scriptfont0 $#1$}\kern-.1em }/
\hbox{\kern-.15em \lower .25ex \hbox{\the\scriptfont0 $#2$}}
}

\title[An upper bound on the number of rational points...]{An upper bound on the number of rational points of arbitrary projective varieties over finite fields}

\author{Alain Couvreur}

\address{INRIA \& LIX, UMR 7161, \'Ecole Polytechnique, 91128 Palaiseau Cedex, France}
\email{alain.couvreur@lix.polytechnique.fr}

\date{\today}

\begin{document}
\maketitle

\thispagestyle{empty}

\begin{abstract}
  We give an upper bound on the number of rational points of an arbitrary
  Zari\-ski closed subset of a projective space over a finite field $\F_q$.
  This bound depends only on the dimensions and degrees
  of the irreducible components and  holds for very general projective
  varieties, even reducible and non equidimensional. As a consequence, we prove
  a conjecture of Ghorpade and Lachaud  on the maximal number of rational points
  of an equidimensional projective variety.
\end{abstract}

\section*{Introduction}
Counting or finding bounds on the number of rational points of a variety
over a finite field is a classical task which arises naturally
in number theory, algebraic geometry or finite geometry.
This problem has motivated
the development of very elegant constructions of cohomology theories
called {\em Weil cohomologies}.
The work of Deligne \cite{deligne} on Weil conjectures
entails among others an upper bound on the number of
rational points of a smooth projective complete intersection $\X$
depending only on $q$ and the Betti numbers of $\X$ for the
\'etale cohomology \cite[Theorem 8.1]{deligne}.

If $\X$ is a smooth curve, this bound is nothing but 
the well-known Weil bound:
$$
|\X (\F_q)| \leq q+1+2g_{\X} \sqrt{q},
$$
where $g_{\X}$ denotes the genus of $\X$.
This bound was actually improved by Serre \cite{Serre_borne}
as
$$
|\X (\F_q)| \leq q+1+g_{\X} \lfloor 2 \sqrt{q} \rfloor.
$$
Subsequently, this bound has been extended to the case of singular curves
in \cite{AP2, aubryperret}.
Another approach due to St{\"o}hr and Voloch \cite{StohrVoloch}
is based on the
use of the Weierstrass points of the curve and provides upper bounds on
the number of rational points of smooth projective curves.
Finally, other upper bounds on the number of points of arbitrary plane curves
arise from purely combinatorial methods,
like Sziklai's bound for arbitrary plane curves \cite{Sziklai}
which was improved by Homma and Kim \cite{HommaKim} using St{\"o}hr--Voloch
bounds. 

For higher dimensional varieties,
Deligne's results hold for arbitrary smooth geometrically irreducible
projective varieties over a finite field.
In addition, a bound ``\`a la Weil'' for irreducible complete intersections
in a projective space is given in \cite{Ghorpade_Lachaud}. 
On the other hand, an upper bound for the number
of rational points of an arbitrary projective hypersurface $\X \subseteq \P^n$
depending only on the degree and dimension of $\X$
has been proved by Serre \cite{lettre} and independently by S{\o}rensen \cite{Sorensen_DiscreteMath}:
$$
|\X (\F_q)| \leq \delta q^{n-1} + \frac{q^{n-1}-1}{q-1},
$$
where $\delta$ denotes the degree of the hypersurface $\X$.
Serre's proof is combinatorial and based on a nice {\em double
counting} argument. Notice that an interesting alternative proof of this
bound has been recently proposed in \cite{DG}.
Further bounds on the number of rational points of
hypersurfaces of fixed dimension and degree and based on such combinatorial methods can be found for instance in \cite{Homma_Kim_hyp, Thas}.


Generalisations of Serre's bound to projective subvarieties
of codimension larger than or equal to $1$ have also been studied
or conjectured. Actually, this problem can be considered
from two different point of views.
One can 
look for bounds on the number of rational points of a subvariety
of $\P^n$ either in terms of the degrees of a family of
defining polynomials or
in terms of its dimension (provided it is equidimensional)
and its degree as a subvariety of $\P^n$.

For the first point of view, a generalisation of Serre's bound has been
conjectured by Tsfasman and Boguslavsky  \cite{Bog}. Given a family
$F_1, \ldots, F_r$ of linearly independent homogeneous polynomials
of total degree $\delta$
in $\F_q [x_0, \ldots, x_m]$, the number of rational points of the corresponding
variety is conjectured to be bounded above by a quantity depending only
on $r, m$ and $\delta$.
This conjecture has been recently proved to be true
for all $r \leq m+1$ and $\delta < q-1$ by Datta and Ghorpade \cite{DG, DG2}.

For the second point of view,
Ghorpade and Lachaud raised a conjecture
\cite[Conjecture 12.2]{Ghorpade_Lachaud}
on the maximal number of rational points of an arbitrary complete intersection
in $\P^n$ depending only on its dimension, its
degree and the dimension of its ambient projective space.
This conjectural upper bound coincides with Serre's bound \cite{lettre} when the
variety is a hypersurface. This conjecture has been discussed 
more recently, in a survey paper of Lachaud and Rolland
\cite[Conjecture 5.3]{LR}.

In the present paper, we prove a general upper bound for arbitrary
projective subvarieties of $\P^n$ possibly non equidimensional.
This bound depends only on $n$, and the degrees and dimensions
of the irreducible components of the variety.
From this bound and
considering the equidimensional case, we prove Ghorpade and Lachaud's
conjecture \cite[Conjecture 12.2]{Ghorpade_Lachaud}. This proves in particular
that Ghorpade and Lachaud's conjecture holds for equidimensional
projective varieties even if they are not complete intersection.

Our bound is proved by purely combinatorial methods
inspired by Serre's proof in \cite{lettre}. The context is however
more difficult since Serre's proof for hypersurfaces
consists in studying an incidence structure involving the intersections
of this hypersurface with hyperplanes. The point is that, for a hypersurface
$\mathcal{X}$ and a hyperplane $\mathcal{H}$ such that $\mathcal{H}
\nsubseteq \mathcal{X}$, then $\mathcal{H} \cap \mathcal{X}$
has codimension $1$ in $\mathcal{H}$.
In particular, no
irreducible component of the hypersurface is contained in the hyperplane.
In the case of an arbitrary variety,
the situation gets harder since some irreducible components can be
 contained in a hyperplane.
For this reason, our proof treats separately the case of a variety with no
irreducible components contained in a hyperplane and varieties having
some irreducible components which are contained in hyperplanes.

Finally, we discuss the sharpness of this bound.
In particular, we prove that in the equidimensional case, this bound is reached by some arrangements of linear varieties which we call \emph{flowers}.
We also leave as open questions, some further possible improvements.

\section{Notation and definitions}\label{sec:nota}

\subsection{Schemes and varieties}
In this article, we fix a finite field $\F_q$.
We denote respectively by $\A^n$ and $\P^n$ the affine and projective
space of dimension $n$ over $\F_q$ defined as
$$
\A^n = \textrm{Spec}\ \F_q [x_1, \ldots, x_n]\qquad
{\rm and}\qquad \P^n = \textrm{Proj}\ \F_q [x_0, \ldots , x_n].
$$
The dual of $\P^n$, which is the variety of hyperplanes of $\P^n$ is
denoted as $\check{\P}^n$.
Given a closed subscheme $\mathcal{S}$ of $\A^n$ or $\P^n$, we denote
by $\mathcal{S}_{\red}$ the reduced scheme supporting $\mathcal{S}$.

In this article, 
an {\em affine  variety}
(resp. {\em projective variety}) denotes a closed reduced subscheme of $\A^n$
(resp. $\P^n$). Hence, such a scheme is always defined over $\F_q$.
In particular, in what follows and unless otherwise specified,
whenever we speak about a hyperplane or a
linear subvariety of $\A^n$ or $\P^n$ it is always defined over $\F_q$.
Finally, given two closed subvarieties $\X , \Y \subseteq \A^n$ or $\P^n$,
by the {\em intersection} $\X \cap \Y$ we always mean the {\em scheme theoretic}
intersection.

\subsection{Irreducibility equidimensionality}
An affine (resp. projective) variety is said to be {\em irreducible} if it is an integral scheme (see \cite[Chapter II.3]{H}),
or equivalently if its defining ideal is prime
in $\F_q[x_1, \ldots, x_n]$ (resp. $\F_q [x_0, \ldots, x_n]$).
We emphasise that, in what follows, unless otherwise
specified, ``irreducible'' means
``irreducible over $\F_q$''. In particular, an irreducible variety
needs not be absolutely irreducible.


A variety is said to be {\em equidimensional} if its irreducible components
all have the same dimension. A closed subscheme $\mathcal{S}$ of $\A^n$ or
$\P^n$ is said to be equidimensional if $\mathcal{S}_{\red}$ is.

\subsection{Degree}
The {\em degree} of a closed equidimensional subscheme of
dimension $d$ of $\P^n$ is $d!$ times the leading coefficient of its
Hilbert polynomial (for instance, see \cite[Chapter III.3]{EH}).
The degree of a closed subscheme of $\A^n$ is defined as the degree of its
projective closure.
For a more geometrical point of view, the degree
of an equidimensional projective variety $\X \subseteq \P^n$ is
the maximum possible number of points of a zero-dimensional
intersection of $\X \otimes {\overline{\F}_q}$ with a
linear subvariety of codimension $\dim \X$ in $\P^n \otimes {\overline{\F}_q}$.

\subsection{Dimension and degree sequences}
A general subvariety $\X \subseteq \A^n$ or $\P^n$ may be reducible and
non equidimensional. Therefore, it has an {\em irredundant decomposition}~:
$$
\X = \X_1\cup \X_2 \cup \cdots \cup \X_r
$$
where $\X_1, \ldots , \X_r$ are irreducible varieties
and for all $i \neq j$, $\X_i \nsubseteq \X_j$.
For all $i \in \{1, \ldots, r\}$ the integers $d_{\X_i}$
and $\delta_{\X_i}$ denote respectively the dimension and the degree of $\X_i$.
The sequences $(d_{\X_i})_{i\in \{1, \ldots ,r\}}$ and ${(\delta_{\X_i})}_{i \in \{1, \ldots, r\}}$ are referred to as the {\em dimension sequence} and the
{\em degree sequence} of $\X$.
It is worth noting that the $d_{\X_i}$'s need not be distinct.
For instance, when $\X$ is equidimensional, the $d_{\X_i}$'s are all equal.
Obviously, for all $i$, we have
$d_{\X_i} \geq 0$ and $\delta_{\X_i} \geq 1$.

More generally, a closed subscheme $\mathcal{S} \subseteq \A^n$ or $\P^n$
has an irredundant decomposition:
$$
\mathcal{S} = \mathcal{S}_1 \cup \cdots \cup \mathcal{S}_r
$$
where ${(\mathcal{S}_1)}_{\red}, \ldots , {(\mathcal{S}_r)}_{\red}$
are the irreducible components of the irredundant decomposition
of $\mathcal{S}_\red$.
The dimension and degree sequences of $\mathcal{S}$
are defined in a similar fashion.
Notice that $\mathcal{S}$ and $\mathcal{S}_\red$ have the same dimension
sequence while for all $i$, $\deg \mathcal{S}_i \geq
\deg{(\mathcal{S}_i)}_{\red}$ with equality if and only if $\mathcal{S}_i$
is reduced.

Finally, we always denote by $D_{\mathcal{S}}$
the maximum dimension of a component of the irredundant decomposition
of $\mathcal{S}$, that is
$$
D_{\mathcal{S}} \eqdef \max \{d_{\mathcal{S}_1}, \ldots, d_{\mathcal{S}_r}\}.
$$


\subsection{Rational points}
A closed point of a scheme over $\F_q$ is said to
be {\em rational} or {\em $\F_q$--rational}
if its residue field is $\F_q$.
The set of rational points of a scheme $\Y$
is denoted as $\Y(\F_q)$.
When $\Y = \P^n$, we denote its number of points by:
$$
\pi_n \eqdef |\P^n (\F_q)| = \frac{q^{n+1}-1}{q-1} \cdot
$$
Moreover, for convenience sake, we set
$$
\pi_j \eqdef 0, \quad {\rm for\ all}\ j<0.
$$
Let us recall that
\begin{equation}
  \label{eq:elem}
  \forall n \geq 0,\qquad \pi_n = q\pi_{n-1} +1,
\end{equation}
which straightforwardly entails
\begin{equation}
  \label{eq:bidon}
  \forall k \geq \ell \geq 0,\quad \pi_{k}-\pi_{\ell} = q(\pi_{k-1}-\pi_{\ell -1}).
\end{equation}

\section{The affine case}

\begin{thm}\label{thm:affine}
  Let $\X \subseteq \A^n$ be an affine variety. Let $d_{\X_1}, \ldots, d_{\X_r}$ be its
dimension sequence and $\delta_{\X_1}, \ldots, \delta_{\X_r}$ its degree sequence.
Then, we always have
$$
|\X(\F_q)| \leq \sum_{i=1}^r \delta_{\X_i}q^{d_{\X_i}}.
$$
\end{thm}

\begin{proof}
  The equidimensional case is proved in \cite{LR}.
  It holds in particular for irreducible varieties.
  Thus, for all $i\in \{1, \ldots ,r\}$,
  we have : 
  $$
  |\X_i(\F_q)| \leq \delta_{\X_i} q^{d_{\X_i}}.
  $$
  The result is obtained by summing up all these inequalities.
\end{proof}

\section{The projective case}


The main result of this article is stated below. Section~\ref{sec:proof_main}
is devoted to its proof.

\begin{thm}\label{thm:main}
  Let $\X \subseteq \P^n$ be a projective variety
  with dimension sequence $d_{\X_1}, d_{\X_2}, \ldots , d_{\X_r}$ with
  $d_{\X_i}<n$ for all $i\in \{1, \ldots , r\}$, and degree sequence
  $\delta_{\X_1}, \ldots , \delta_{\X_r}$.
Then,
\begin{equation}
  \label{eq:main_thm}
|\X(\F_q)| \leq \left(\sum_{i=1}^r \delta_{\X_i}(\pi_{d_{\X_i}} - \pi_{2d_{\X_i} - n}) \right) + \pi_{2D_{\X} -n},
\end{equation}
where $D_{\X} = \max \{d_{\X_1}, \ldots , d_{\X_r}\}$.
\end{thm}

\begin{rem}
  \label{rem:var_to_scheme}
  Theorem~\ref{thm:main} holds actually when replacing $\X$ by
  a closed subscheme $\mathcal{S}$ of $\P^n$.
  Indeed, the result applies to the variety $\mathcal{S}_\red$. Then,
  notice that
  \begin{enumerate}[(i)]
  \item $\mathcal{S}(\F_q) = \mathcal{S}_{red}(\F_q)$;
  \item for all scheme $\mathcal{S}_i$ of the irredundant decomposition
    of $\mathcal{S}$, we have $\deg \mathcal{S}_i \geq \deg
    {(\mathcal{S}_i)}_{\red}$ and $\dim \mathcal{S}_i =
    \dim {(\mathcal{S}_i)}_{\red}$.
  \end{enumerate}
  Therefore, the right hand side of~(\ref{eq:main_thm})
  is smaller when applied to the irredundant decomposition of
  $\mathcal{S}_\red$ than when applied to that of $\mathcal{S}$.
\end{rem}

Ghorpade and Lachaud's conjecture \cite[Conjecture 12.2]{Ghorpade_Lachaud}
is a straightforward
corollary of Theorem \ref{thm:main} since it is noting but the equidimensional
case.

\begin{cor}\label{cor:Lachaud}
  Let $\X\subseteq \P^n$ be an equidimensional projective variety of
  dimension $d<n$ and degree $\delta$. Then,
  $$
  |\X(\F_q)| \leq \delta(\pi_d - \pi_{2d-n})+ \pi_{2d-n}.
  $$
\end{cor}

\begin{rem}
  Actually Ghorpade and Lachaud stated this conjecture under the additional
  hypothesis ``$\X$ is a complete intersection''. The conjecture
  turns out to be true even without this hypothesis.
\end{rem}

\section{The proof}\label{sec:proof_main}

We will prove Theorem~\ref{thm:main} by induction on the dimension $n$ of the
ambient space. 
Firstly, let us introduce a notation.

\begin{nota}\label{nota:B}
  Let $\mathcal{S} \subseteq \P^n$ be a closed subscheme of $\P^n$
  with dimension and degree sequences $d_1, \ldots, d_s$ and
  $\delta_1, \ldots, \delta_s$. We define $B_n(\mathcal{S})$ as: 
  $$
  B_n (\mathcal{S})
  \eqdef \left( \sum_{i=1}^s \delta_i \left( \pi_{d_i} - \pi_{2d_i -n}
    \right)\right) + \pi_{2D_\mathcal{S} -n}.
  $$
\end{nota}

\subsection{A consequence of B\'ezout theorem}

The following classical statement is central in the proofs to follow.
It is nothing but a corollary of a refined version of Bézout Theorem (see for instance \cite{Fulton_intersection}). We outline an ad hoc proof for the
comfort of the reader.

\begin{prop}
  \label{prop:hyperplane_section}
  Let $\X \subseteq \P^n$ be an irreducible projective variety of dimension
  $d\geq 1$ and degree $\delta$ and $\H$ be a hyperplane of $\P^n$ which does
  not contain $\X$. Then, $\X \cap \H$
  is an equidimensional scheme of dimension $d-1$ and degree $\delta$.
\end{prop}

\begin{proof}
  Since $\X \nsubseteq \H$, an irreducible component of $(\X \cap \H)_{\red}$
  has dimension $<d$, while from \cite[Theorem 7.2]{H}
  it has dimension $\geq d-1$. Thus, $\X \cap \H$ is equidimensional
  of dimension $d-1$. For the degree, the exact sequence
  $$
  \xymatrix{0 \ar[r]& \fract{\F_q[x_0,\ldots, x_n]}{I} \ar[r]^{\times f}&
    \fract{\F_q[x_0,\ldots, x_n]}{I} \ar[r]&
    \fract{\F_q[x_0,\ldots, x_n]}{I+(f)} \ar[r] & 0
    }
  $$
  entails a relation on the Hilbert polynomials
  $P_\X$ and $P_{\X \cap \H}$ of $\X$ and $\X \cap \H$. Namely~:
  $$
  P_{\X \cap \H}(t) = P_{\X}(t) - P_{\X}(t-1). 
  $$
  Since the leading term of $P_\X$ is $\frac{\delta}{d!}$, after an easy
  computation, we see that the leading term of $P_{\X \cap \H}$
  equals $\frac{\delta}{(d-1)!}$, which concludes the proof.
\end{proof}

\begin{cor}
  \label{cor:central}
  Let $\X \subseteq \P^n$ be a projective variety
  with irredundant decomposition $\X_1 \cup \cdots \cup \X_s$.
  Let $\H \in \check{\P}^n(\F_q)$ be a hyperplane which does not contain
  any of the $\X_i$'s. Then,
  $$
  B_{n-1} (\X \cap \H) = \left(\sum_{i=1}^s \delta_{\X_i} \left(\pi_{d_{\X_i} - 1} -
      \pi_{2d_{\X_i} -n - 1} \right)\right) + \pi_{2D_{\X} - n - 1}.
  $$
  In particular, this quantity is independent from $\H$ and denoted
  $B_{n-1}(\X \cap \ \cdot\ )$.
\end{cor}

\begin{proof}
  For all $i \in \{1, \ldots , s\}$, define
  $\Y_{i,1}, \ldots, \Y_{i,s_i}$ the components of the irredundant decomposition
  of $\X  \cap \H$. We have
  \begin{align*}
    B_{n-1}(\X \cap \H) & = \left(\sum_{i=1}^s \sum_{j = 1}^{s_i} \delta_{\Y_{i,j}}
    \left(\pi_{d_{\Y_{i,j}}} - \pi_{2d_{\Y_{i,j}}-(n-1)} \right)\right)
    +\pi_{2D_{\X \cap \H}-(n-1)}. 
  \end{align*}
From Proposition~\ref{prop:hyperplane_section},
  for all $i, j$, we have $d_{\Y_{i,j}} = d_{\X_i -1}$. Hence,
  $D_{\X \cap \H} = D_{\X} - 1$ and
  \begin{align*}
     B_{n-1}(\X \cap \H)
                  & = \left(\sum_{i=1}^s \sum_{j = 1}^{s_i} \delta_{\Y_{i,j}}
    \left(\pi_{d_{\X_i}-1} - \pi_{2(d_{\X_i}-1)-(n-1)}\right)  \right)
    + \pi_{2(D_{\X}-1)-(n-1)} \\
                  & = \left(\sum_{i=1}^{s} \left(\pi_{d_{\X_i}-1} -\pi_{2d_{\X_i}-n-1}
                    \right) \sum_{j=1}^{s_i} \delta_{\Y_{i,j}}\right)
                  + \pi_{2D_{\X} - n -1} 
  \end{align*}
  For all $i$, we have $\sum_{j=1}^{s_i} \delta_{\Y_{i,j}} = \delta_{\X_i \cap \H}$
  which equals $\delta_{\X_i}$ from Proposition~\ref{prop:hyperplane_section}.
  This concludes the proof.
\end{proof}

\subsection{A remark on the zero dimensional part}\label{ss:zerodim}
If $\X = \Y \cup \Zc$, where $\Y$ is a union of irreducible
components of dimension larger than or equal to $1$ and $\Zc$
has dimension
$0$, then it is sufficient to prove the upper bound on $\Y$. Indeed,
denote by $\delta_{\Zc}$ the degree of $\Zc$ and
by $\Y_1, \ldots, \Y_r$ the irreducible components if $\Y$.
Then, notice that $D_{\X} = D_{\Y}$. The upper bound we wish to prove becomes
$$
|\X (\F_q)| \leq \sum_{i=1}^r \delta_{\Y_i}(\pi_{d_{\Y_i}} - \pi_{2d_{\Y_i}-n}) + \delta_{\Zc} + \pi_{2D_{\Y} -n}. 
$$ 
Thus, if we can prove the result for $\Y$, then since 
$$
|\Zc(\F_q)|\leq \delta_{\Zc} \qquad \textrm{and}\qquad
|\X(\F_q)| \leq |\Y (\F_q)| +|\Zc(\F_q)|,
$$
we get the result.
Therefore, from now on, we assume that $\X$ has no zero-dimensional
component. That is to say, the dimension sequence satisfies
$$
\forall i \in \{1, \ldots, r\},\quad d_{\X_i}> 0.
$$

\subsection{Initialisation}
The case $n=1$, which corresponds to that of a zero-dimensional subscheme
of the projective line is obvious.
Actually, the case $n=2$ is already known.
Indeed, as suggested in \S \ref{ss:zerodim}, one can assume
that $\X$ has no zero-dimensional component and hence is a curve of degree 
$\delta$. Then, the upper bound
$$
|\X (\F_q)| \leq \delta q+1
$$
is a direct consequence of \cite{lettre}.

\begin{rem}
  Notice that some refined bounds on the number of points of plane curves
  are given in \cite{HommaKim, StohrVoloch, Sziklai}.
\end{rem}

\subsection{The induction step under different assumptions}
From now on, we assume that $n\geq 3$.
Consider the two following assumptions.

\begin{ass}\label{ass:step1}
No irreducible component of $\X$ is contained in a hyperplane.
\end{ass}

\begin{ass}\label{ass:step3}
Every irreducible component of $\X$ is either linear or is not contained in 
any hyperplane.
\end{ass}

The induction step is proved under Assumption~\ref{ass:step1} in \S
\ref{ss:ass1}, then under Assumption \ref{ass:step3} in \S~\ref{ss:ass3}.
Finally, the general case is treated in \S \ref{ss:general}.



\subsection{Proof under Assumption \ref{ass:step1}}\label{ss:ass1}
First, notice that the upper bound in Theorem~\ref{thm:main}
is obviously true if $\X (\F_q) = \emptyset$.
Therefore, assume from now on that $\X (\F_q)$ is nonempty and
let $P\in \X(\F_q)$. By Assumption~\ref{ass:step1},
no $\H \in \check{\P}^n(\F_q)$ containing $P$ contains an
irreducible component of $\X$.
Next, let us introduce the  bipartite graph $\G$ whose first and second
vertex sets are
\begin{align*}
\V_1 & \eqdef \X(\F_q)\setminus \{P\},\\
\V_2 & \eqdef \{\H \in \check{\P}^n(\F_q) ~|~  P \in \H\}
\end{align*}
and whose edge set is
$$
\E \eqdef \left\{(Q, \H) \in \V_1 \times \V_2 ~|~ Q\in \H \right\}.
$$
The heart of the proof consists in counting the set of edges by two distinct manners. 

In what follows, the number of edges containing a given vertex
is referred to as the {\em valency} of the vertex.

\begin{rem}
  Usually, in graph theory, the number of edges containing a given vertex
  is referred to a the {\em degree} of the vertex.
  We chose {\em valency} to avoid confusions with the notion of degree of a
  subscheme of $\A^n$ or $\P^n$.
\end{rem}

Let us summarise some properties of the graph:
\begin{enumerate}[(i)]
\item\label{item:V2} $|\V_2| = \pi_{n-1}$;
\item\label{item:degV1} the valency of a vertex $Q \in \V_1$
equals $\pi_{n-2}$
\item\label{item:degV2} the valency of a vertex
  $\H \in \V_2$ equals $|(\X \cap \H)(\F_q) \setminus \{P\}|$ which,
  by induction,
  is bounded above by $B_{n-1}(\X \cap \H ) -1$
  (see Notation~\ref{nota:B}) and hence by $B_{n-1}(\X \cap \ \cdot \ )-1$
  (see Corollary~\ref{cor:central}).
  From Remark~\ref{rem:var_to_scheme},
  this holds even if $\X \cap \H$ is non reduced.
\end{enumerate}

First, assume that $B_{n-1}(\X \cap \ \cdot \ ) \geq \pi_{n-1}$.
From Corollary~\ref{cor:central}, we have
$$
 \left( \sum_{i=1}^r \delta_{\X_i} (\pi_{d_{\X_i}-1} - \pi_{2d_{\X_i}-n-1})\right)
  + \pi_{2D_{\X}-n-1}  \geq  \pi_{n-1},
$$
then, multiplying both sides by $q$ yields (thanks to (\ref{eq:bidon}))
\begin{equation}
\label{eq:easy} \left( \sum_{i=1}^r \delta_{\X_i} (\pi_{d_{\X_i}} - \pi_{2d_{\X_i}-n})\right)
  + \pi_{2D_{\X}-n}  \geq   \pi_n .
\end{equation}
Since $\X \subseteq \P^n$, we have $|\X (\F_q)| \leq \pi_n$.
Therefore, from (\ref{eq:easy}), the result is straightforward.

From now on, we assume that
\begin{equation}\label{eq:B}
B_{n-1}(\X \cap \ \cdot \ ) < \pi_{n-1}.
\end{equation}
From (\ref{item:degV1}), we have:
\begin{equation}
  \label{eq:degV1}
 |\E| = |\mathcal{V}_1| \pi_{n-1} = (|\X(\F_q)| -1)\pi_{n-2}.
\end{equation}
From (\ref{item:V2}) and (\ref{item:degV2}), we have
\begin{equation}
  \label{eq:degV2}
  |\E| \leq \pi_{n-1} 
   \left(B_{n-1}(\X \cap  \ \cdot \ )-1\right).
\end{equation}
This yields
\begin{equation}
  \label{eq:cardX}
  |\X (\F_q)| \leq 1+ \frac{\pi_{n-1}}{\pi_{n-2}} (B_{n-1}(\X \cap \ \cdot\ )-1)
 \leq 1+ \left(q+\frac{1}{\pi_{n-2}}\right) (B_{n-1}(\X \cap \ \cdot\ )-1).
\end{equation}
By the definition of $B_{n-1}(\X \cap \ \cdot\ )$ and,
thanks to (\ref{eq:bidon}), we get
$$
\begin{aligned}
  |\X (\F_q)|
     &\leq \left(\sum_{i=1}^r \delta_{\X_i}(\pi_{d_{\X_i}} - \pi_{2d_{\X_i} -n})\right) +
     \pi_{2D_{\X}-n}  - q + \frac{B_{n-1}(\X \cap \ \cdot \ )-1}{\pi_{n-2}}\cdot
\end{aligned}
$$
From (\ref{eq:B}), we have
$\frac{B_{n-1}(\X \cap \ \cdot\ )-1}{\pi_{n-2}} < \frac{\pi_{n-1} -1}{\pi_{n-2}}$
and, using (\ref{eq:elem}),
we obtain $\frac{B_{n-1}(\X \cap \ \cdot\ )-1}{\pi_{n-2}} < q$, which yields
$$
  |\X (\F_q)|
     \leq \left(\sum_{i=1}^r \delta_{\X_i}(\pi_{d_{\X_i}} - \pi_{2d_{\X_i} -n})\right) +
     \pi_{2D_{\X}-n}.
$$

\begin{rem}
  Instead of considering an incidence graph, the proof can be realised
using a purely algebraic geometric point of view by defining the incidence
variety
$$
\mathcal{T} \eqdef \left\{(Q, \H) \in \X \times \check{\P}^n ~|~ Q\in \H\right\}.
$$
Then, the approach consists in counting the number of rational points
of $\mathcal{T}$ by two different manners by estimating the
number of rational points of the fibres of the canonical projections
$\mathcal{T} \rightarrow \check{\P}^n$ and $\mathcal{T} \rightarrow \X$.
This point of view is developed in \cite[\S 12]{Ghorpade_Lachaud}
and \cite[\S 2]{LR}.
\end{rem}

\subsection{Proof under Assumption \ref{ass:step3}}
\label{ss:ass3}
The proof under Assumption~\ref{ass:step3} is done
in three steps. Under the assumption that no nonlinear irreducible component
of $\X$ is contained in a $\F_q$--rational hyperplane:
\begin{enumerate}
\item\label{it:case_hyperplane} we first treat the case when one of the
  irreducible components of $\X$ is a hyperplane;
\item then, we treat the case when only one irreducible component $\L$ of
  $\X$ is linear and $\L$ is not a hyperplane;
\item finally, we treat the case of multiple linear irreducible components
which are not hyperplanes.
\end{enumerate}

The treatment of the first step will require the following lemma.

\begin{lem}\label{lem:for_the_referee}
  Let $\mathcal{Y}$ be an irreducible closed
  sub-variety of $\P^n$   of dimension $d$ and degree $\delta$. Let
  $\mathcal{H} \in \check{\P}^n(\F_q)$ and
  $\mathcal{Y}_{\textrm{aff}}$ be the affine chart
  $\mathcal{Y} \setminus \mathcal{H}$ of $\mathcal{Y}$. Then,
  $\mathcal{Y}_{\textrm{aff}}$   has degree $d$ and dimension $\delta$. 
\end{lem}

\begin{proof}
  Since $\mathcal{Y}$ is irreducible, its dimension can be defined
  as the transcendence degree of its function field over $\F_q$.
  Since the function field of every open subset of $\mathcal{Y}$
  equals that of $\mathcal{Y}$, we deduce that $\mathcal{Y}$
  and $\mathcal{Y}_{\textrm{aff}}$ have the same dimension.

  The variety $\mathcal{Y}_{\textrm{aff}}$ is affine, its degree
  is that of its projective closure (see \S \ref{sec:nota}),
  which is nothing but $\mathcal{Y}$. 
\end{proof}


\subsubsection{If $\X$ has a hyperplane in its irredundant decomposition}
Let $\H \in \check{\P}^n(\F_q)$ be this hyperplane. Write the irredundant
decomposition of $\X$ as a union of irreducible varieties~:
$$\X = \X_1 \cup \X_2 \cup \cdots \cup \X_r,\qquad \textrm{with}\quad \X_1 = \H
$$
and such that for all $i \neq j$, $\X_i \nsubseteq \X_j$.
Notice that $D_{\X} = d_{\H} = n -1$.
Now, $\X$ is the disjoint union of $\H$ and
$\X_{\rm aff} \eqdef \X \setminus \H$.
The variety $\X_{\rm aff}$ is affine and, from Lemma~\ref{lem:for_the_referee},
its dimension and degree sequences are
$d_{\X_2}, \ldots , d_{\X_r}$ and $\delta_{\X_2}, \ldots , \delta_{\X_r}$.
Thanks to Theorem~\ref{thm:affine}, we get
\begin{equation}\label{eq:affine}
|\X (\F_q)| = |\H(\F_q)| +|\X_{\rm aff}(\F_q)|\leq
\pi_{n-1} + \sum_{i=2}^r \delta_{\X_i} q^{d_{\X_i}}.
\end{equation}
Next, notice that for all $i$, we have
\begin{equation}\label{eq:qk}
q^{d_{\X_i}} \leq \pi_{d_{\X_i}} - \pi_{2d_{\X_i}-n}.
\end{equation}
Indeed, since, $d_{\X_i} < n $, we have $2d_{\X_i} -n < d_{\X_i}$.
On the other hand, since $\delta_{\H} = 1$, we get
\begin{equation}\label{eq:d1}
\pi_{n-1} = \pi_{d_{\H}} = \delta_{\H} (\pi_{d_{\H}} - \pi_{2d_{\H} - n}) + \pi_{2d_{\H} - n}.
\end{equation}
Putting (\ref{eq:affine}), (\ref{eq:qk}) and (\ref{eq:d1}) together,
we get
$$
|\X (\F_q)| \leq \left(\delta_{\H}(\pi_{d_{\H}} - \pi_{2d_{\H} - n}) + \sum_{i=2}^r \delta_{\X_i} (\pi_{d_{\X_i}} - \pi_{2d_{\X_i}-n})\right) +
\pi_{2d_{\H} -n}
$$
and, since $d_{\H} = D_{\X}$, this yields the result.

\subsubsection{When, $\X$ has a single linear subvariety
which is not a hyperplane in its irredundant decomposition}\label{ss:notahyp}
Let $\L$ be this linear subvariety of $\P^n$ which is
not a hyperplane. That is, its dimension satisfies $d_{\L} < n-1$.
Moreover, we assume that the other irreducible components of $\X$
are non linear. Recall that, by assumption, none of the other components
is contained in a hyperplane.
Here again, we write the irredundant decomposition of $\X$ as a union of 
irreducible components as
$$
\X  = \X_1 \cup \X_2 \cup \cdots \cup \X_r,\qquad \textrm{with}\quad \X_1 = \L.
$$

As for the proof under Assumption~\ref{ass:step1}, 
we will apply a combinatorial proof based on another incidence
structure. The incidence graph we consider is obtained as follows.
Choose $P\in \L (\F_q)$ and set
$$
\begin{aligned}
  \V_1 &\eqdef (\X\setminus \L) (\F_q);\\
  \V_2 &\eqdef \left\{\H \in \check{\P}^n(\F_q) ~|~  P\in \H\ \ \textrm{and}\ \
  \L \nsubseteq \H \right\};\\
  \E & \eqdef \{(Q, \H)\in \V_1 \times \V_2 ~|~ Q\in \H\}.
\end{aligned}
$$
Here, we have
$$
  |\V_2|  = \pi_{n-1} - \pi_{n-d_{\L}-1}.
$$
Indeed, it suffices to notice that the set of hyperplanes
in $\check{\P}^n(\F_q)$ containing
$\L$ has $\pi_{n-d_{\L}-1}$ elements.
Next, notice that in this incidence graph:
\begin{itemize}
\item the valency of a vertex of $\V_1$ is $\pi_{n-2}-\pi_{n-d_{\L}-2}$;
\item by induction and since no irreducible
component of $\X$ but $\L$ is contained in a hyperplane,
the valency of a vertex $\H$ of $\V_2$
equals $|(\X \cap \H)(\F_q)| - |(\L \cap \X)(\F_q)|$ which,
from Corollary~\ref{cor:central} is bounded above by
$$\sum_{i=1}^r \left( \delta_{\X_i} (\pi_{d_{\X_i}-1} - \pi_{2d_{\X_i} - n -1}) \right)
+ \pi_{2D_{\X} -n -1} - \pi_{d_{\L}-1},$$
where the ``$-\pi_{d_{\L}-1}$'' term corresponds to the
rational points of $\L \cap \H$, which are not counted.
\end{itemize}
Now, as in \S \ref{ss:ass1}, by counting the number of edges
of the graph in two different manners, we get
$$
\begin{aligned}
\left(|\X (\F_q)| - \pi_{d_{\L}} \right)(\pi_{n-2} & - \pi_{n-d_{\L}-2})
\leq\\
(\pi_{n-1} - &\pi_{n-d_{\L}-1})
\left[ \left( \sum_{i=1}^r \delta_{\X_i} (\pi_{d_{\X_i}-1} -
\pi_{2d_{\X_i} -n - 1}) \right) + \pi_{2D_{\X} - n -1} - \pi_{d_{\L}-1}\right].
\end{aligned}
$$
Using (\ref{eq:bidon}), we get
$$
\begin{aligned}
|\X (\F_q)| &  \leq  \pi_{d_{\L}} + q\left[ \left( \sum_{i=1}^r \delta_{\X_i} (\pi_{d_{\X_i}-1} -
\pi_{2d_{\X_i} -n - 1}) \right) + \pi_{2D_{\X} - n -1} - \pi_{d_{\L}-1}\right]\\
  & \leq \pi_{d_{\L}} + \left( \sum_{i=1}^r \delta_{\X_i} (\pi_{d_{\X_i}} -
\pi_{2d_{\X_i} -n }) \right) + \pi_{2D_{\X} - n } - \pi_{d_{\L}}\\
 & \leq \left( \sum_{i=1}^r \delta_{\X_i} (\pi_{d_{\X_i}} -
\pi_{2d_{\X_i} -n }) \right) + \pi_{2D_{\X} - n }.
\end{aligned}
$$

\subsubsection{When there are several linear components}

Now assume that $\X$ contains more than one linear irreducible 
component. Its irredundant decomposition is
$$
\X = \Y \cup \L_1 \cup \cdots \cup \L_s
$$
where $s>1$ and $\Y$ has no linear irreducible component.
Moreover, recall that we are still under Assumption~\ref{ass:step3},
that is to say that none of the irreducible
components $\Y_1, \ldots, \Y_r$ of $\Y$
is contained in a hyperplane of $\P^n$.
Assume that the linear components $\L_1, \ldots , \L_s$ are
sorted by decreasing dimensions $d_{\L_1} \geq d_{\L_2} \geq \cdots \geq
d_{\L_s}$.
From \S \ref{ss:notahyp}, we have
$$
|(\Y \cup \L_1)(\F_q)|\leq \left( \sum_{i=1}^r \delta_{\Y_i} (\pi_{d_{\Y_i}} - \pi_{2d_{\Y_i}})\right) + (\pi_{d_{\L_1}} - \pi_{2d_{\L_1}-n})+ \pi_{2D_{\X} - n}.
$$
Now, notice that 
\begin{equation}\label{eq:grassman}
|\X (\F_q)| \leq |(\Y \cup \L_1)(\F_q)| +
\sum_{i=2}^s \left( |\L_i (\F_q)| - |(\L_1 \cap \L_i) (\F_q)| \right).
\end{equation}
Moreover, for all $i\in \{2, \ldots , s\}$
\begin{equation}
  \label{eq:L1capL2}
  \dim \L_1 \cap \L_i \geq d_{\L_1} + d_{\L_i} -n \geq 2d_{\L_i} -n
\end{equation}
Putting (\ref{eq:grassman}) and (\ref{eq:L1capL2}) together, we get
$$
| \X (\F_q)| \leq \left( \sum_{i=1}^r \delta_{\Y_i} (\pi_{d_{\Y_i}} - \pi_{2d_{\Y_i}})\right) + (\pi_{d_{\L_1}} - \pi_{2d_{\L_1}-n})+ \pi_{2D_{\X} - n} +
\sum_{i=2}^s (\pi_{d_{\L_i}} - \pi_{2d_{\L_i}-n}).
$$
And hence,
$$
|\X (\F_q)| \leq \left( \sum_{i=1}^r \delta_{\Y_i} (\pi_{d_{\Y_i}} - \pi_{2d_{\Y_i}})\right) + \left( \sum_{i=1}^s
(\pi_{d_{\L_i}} - \pi_{2d_{\L_i}-n}) \right)+ \pi_{2D_{\X} - n},
$$
which yields the expected upper bound.

\subsection{Proof in the general case}\label{ss:general}
Now, assume that 
$$\X = \Y \cup \Zc =  \Y_1 \cup \cdots \cup \Y_{s} \cup \Zc_1 \cup \cdots \cup \Zc_{\ell}$$
where the $\Y_i$'s are irreducible varieties which are either linear
or are not contained in any hyperplane in $\check{\P}^n(\F_q)$ and
the $\Zc_i$'s are non linear and each one is contained in at least one
hyperplane in $\check{\P}^n(\F_q)$.
In particular, since the $\Zc_i$'s are non linear,
for all $i\in \{1, \ldots, \ell\}$, we have $\delta_{\Zc_i} >1$.
From the previous results, and since we clearly have $D_\Y \leq D_\X$,
 we already know that
\begin{equation}\label{eq:Ybound}
|\Y (\F_q)| \leq \left( \sum_{i=1}^r \delta_{\Y_i}
(\pi_{d_{\Y_i}} - \pi_{2d_{\Y_i} -n})\right) + \pi_{2D_{\X} -n}.
\end{equation}
Next, for all $i\in \{1, \ldots , \ell\}$, since $\Zc_i$ is contained in 
a hyperplane, one can apply the induction hypothesis to this hyperplane and
get
\begin{equation}\label{eq:Zbound}
|\Zc_i (\F_q)| \leq B_{n-1}(\Zc_i)  = 
\delta_{\Zc_i} (\pi_{d_{\Zc_i}} - \pi_{2d_{\Zc_i} -n + 1}) + \pi_{2d_{\Zc_i} -n+1}.
\end{equation}

\begin{lem}\label{lem:trick}
  For all integers $d > 0$ and $\delta >1$, we have
$$
\delta (\pi_d - \pi_{2d-n+1}) +\pi_{2d-n+1} \leq \delta (\pi_d - \pi_{2d-n}).
$$ 
\end{lem}

\begin{proof}
  Consider the difference,
$$
\begin{aligned}
  \delta (\pi_d - \pi_{2d -n })  -  \bigg(\delta (\pi_d - \pi_{2d -n + 1}) + \pi_{2d -n+1}\bigg)  
&= \delta (\pi_{2d-n+1} - \pi_{2d-n}) -  \pi_{2d-n+1}\\
&= \delta q^{2d-n+1} -\pi_{2d-n+1}\\
&= (\delta -1) q^{2d-n+1}- \pi_{2d-n}.
\end{aligned}
$$
Since $\delta >1$, to prove that this difference is nonnegative,
it is sufficient
to prove that $q^{2d-n+1} \geq \pi_{2d-n}$. It is obviously true if
$2d-n<0$. It also holds true if $2d-n \geq 0$ since, using that $q\geq 2$,
we have
$$
\pi_{2d - n } = \frac{q^{2d-n+1}-1}{q-1} \leq q^{2d-n+1} - 1.
$$
\end{proof}

From Lemma~\ref{lem:trick}, and since by assumption on
the $\Zc_i$'s, we have $\delta_{\Zc_i} >1$ for all $i\in \{1, \ldots, \ell\}$, we
see that
\begin{equation}\label{eq:trivbound}
\forall i \in \{1, \ldots , \ell\},\quad \delta_{\Zc_i} (\pi_{d_{\Zc_i}} - \pi_{2d_{\Zc_i} -n + 1})  + \pi_{2d_{\Zc_i} -n+1}
\leq \delta_{\Zc_i} (\pi_{d_{\Zc_i}} - \pi_{2d_{\Zc_i} -n }).
\end{equation}
Next, use the obvious inequality
$$
|\X (\F_q)| \leq |\Y (\F_q)| + \sum_{j=1}^{\ell} |\Zc_j (\F_q)|
$$
and put it together with the bounds (\ref{eq:Ybound}), (\ref{eq:Zbound}) and (\ref{eq:trivbound}) to get
$$
|\X (\F_q)| \leq \left(\sum_{i=1}^r \delta_{\Y_i} (\pi_{d_{\Y_i}} - \pi_{2d_{\Y_i} -n})\right) +
\left(\sum_{i=1}^r \delta_{\Zc_i} (\pi_{d_{\Zc_i}} - \pi_{2d_{\Zc_i} -n})\right)
+ \pi_{2D_{\X} - n}.
$$
This concludes the proof.

\section{Is this new bound Optimal?}

\subsection{The bound is optimal for equidimensional varieties}\label{ss:optimal}\label{ss:equidim}

We will show that the bound given by Theorem \ref{thm:main} is reached by
equidimensional arrangement of linear varieties.
This shows that Corollary \ref{cor:Lachaud} is an optimal upper bound.
For that we introduce two objects: {\em partial $d$--spreads} and 
{\em $d$--flowers}.
The notion of partial $d$--spread is well--known and subject of intense
study in finite geometry. For instance see  (the list is far from being exhaustive)
\cite{Beutelspacher,Jungnickel1984,Jungnickel1993,Storme}.
On the other hand, the terminology of $d$--flower is introduced by the author.

\begin{defn}[partial $d$--spreads]
  Let $d, n$ be two positive integers with $2d<n$.
  A {\em partial $d$--spread} is a disjoint union of linear
subvarieties of dimension $d$ of $\P^n$.
\end{defn}

\begin{defn}[$d$--flowers]
Let $d, n$ be two positive integers with $d<n$ and $2d\geq n$.
A {\em $d$--flower} is a union $\L_1 \cup \cdots \cup \L_r$
of linear subvarieties of dimension $d$ of $\P^n$ such that there
exists a linear variety $\mathcal{M}$ of dimension $2d-n$
satisfying
$$
\forall i\neq j \in \{1, \ldots, r\},\quad \L_i \cap \L_j = \mathcal{M}.
$$  
\end{defn}

\begin{exmp}
  An $(n-1)$--flower is nothing but a hypersurface obtained as a union
  of hyperplanes meeting at a common $2$--codimensional linear variety.
  These flowers reach Serre's bound for the number of points of hypersurfaces
\cite{lettre}.
\end{exmp}

\begin{exmp}
A union of planes of $\P^4$ meeting at a single point is
a $2$--flower.
\end{exmp}

\begin{prop}
  Let $\X$ be a partial $d$--spread or a $d$--flower of degree $\delta$.
  Then
  $$
  |\X (\F_q)| = \delta (\pi_{d} - \pi_{2d-n}) + \pi_{2d-n}.
  $$
\end{prop}

\begin{proof}
  It is a straightforward consequence of the definition of partial
$d$--spreads and $d$--flowers.
\end{proof}

\subsection{The non equidimensional case might have a sharper bound}
For non equidimensional varieties, the optimality of
our bound is less clear. In particular, the following statement asserts
that the upper bound of Theorem \ref{thm:main}  cannot be reached by
non equidimensional arrangements of linear varieties.

\begin{prop}\label{prop:for_linear}
  Let $\X = \L_1 \cup \cdots \cup \L_r$ be a union of 
linear subvarieties of $\P^n$ defined over $\F_q$
such that $d_{\L_1}\geq \cdots \geq d_{\L_r}$ and $r\geq 2$. 
Then,
$$
|\X (\F_q)| \leq \pi_{d_{\L_1}} +
\left(  \sum_{i=2}^r (\pi_{d_{\L_i}} - \pi_{d_{\L_i}+d_{\L_1} - n})\right). 
$$
\end{prop}

\begin{proof}
  Use the obvious upper bound
$$
|\X(\F_q)| \leq |\L_1(\F_q)| + \sum_{i=2}^r |\L_i (\F_q)| - \sum_{i=2}^r
|(\L_1 \cap \L_i)(\F_q)|
$$
together with the fact that
$$
\dim \L_1 \cap \L_i \geq d_{\L_1}+d_{\L_i} -n.
$$
\end{proof}

\begin{rem}
The construction of arrangements of linear subvarieties reaching this
upper bound can be done as follows:
\begin{enumerate}
\item choose an arbitrary $\L_1 \subseteq \P^n$ linear of dimension $d_{\L_1}$;
\item choose $\L_2$ of dimension $d_{\L_2}$
such that $\dim \L_1 \cap \L_2 = d_{\L_1}+d_{\L_2} - n$;
\item choose the $\L_i$'s so that $\L_1 \cap \L_i \subseteq \L_1 \cap \L_2$.
\end{enumerate}
\end{rem}

\begin{rem}
Compared to the upper bound of Theorem \ref{thm:main}, the above upper
bound is sharper. The difference between the bound of Theorem \ref{thm:main}
and Proposition \ref{prop:for_linear} is:
$$
\left(\sum_{i=1}^r (\pi_{d_{\L_i}} - \pi_{2d_{\L_i} - n}) \right)
+ \pi_{2d_{\L_1} -n} -  \left(\pi_{d_{\L_1}} + \sum_{i=2}^r (\pi_{d_{\L_i}}
- \pi_{d_{\L_i} + d_{\L_1} -n}) \right) = \qquad \qquad \qquad \qquad $$
$$
\qquad \qquad \qquad \qquad \qquad \qquad 
\qquad \qquad \qquad \qquad \qquad \qquad 
\sum_{i=2}^r \left(\pi_{d_{\L_i} + d_{\L_1} -n} - \pi_{2d_{\L_i} -n}\right)
$$
and the above difference is positive as soon as the variety is not
equidimensional. 
\end{rem}

Since in the equidimensional case, the upper bound is
reached by arrangement of linear varieties,
the previous observations on arrangements of linear varieties
suggests a possible sharper upper bound which we leave
as an open question.

\begin{question}
  For a projective variety $\X$ decomposed in a union of
  irreducible components $\X_1 \cup \cdots \cup \X_r$ sorted by decreasing
  dimensions, i.e. $d_{\X_1} \geq \cdots \geq d_{\X_r}$, do we have
$$
|\X(\F_q)| \leq 
\left(  \sum_{i=1}^r \delta_{\X_i}(\pi_{d_{\X_i}} - \pi_{d_{\X_i}+d_{\X_1} - n})\right) ~+~
\pi_{2d_{\X_1} -n} \ 
{\rm ?}
$$
Notice that if the degrees $\delta_{\X_i}$ are all equal to $1$
we get the upper bound of Proposition \ref{prop:for_linear}.
\end{question}

\subsection{An open question on complete intersections}
First, notice that partial spreads and flowers of dimension $<n-1$
are never complete intersections unless they are irreducible and hence of
degree $1$.
Indeed,
from \cite[Ex II.8.4, III.5.5]{H}, complete intersections are always connected
while partial spreads are not. For a flower $\mathcal{F}$,
of dimension $<n-1$ and degree $>1$, 
one can choose a linear variety $\L$ which is transverse to $\mathcal{F}$
and such that the intersection $\L \cap \mathcal{F}$ is a partial spread,
hence is disconnected. If $\mathcal{F}$ was a complete intersection, 
then $\mathcal{F}\cap \L$ would be a complete intersection too, which is
a contradiction.

Second, in \cite[\S 5]{LR}, the authors study some arrangements
of linear varieties
which are complete intersection and have a large number of points.
They call these varieties {\em tubular sets}.
They prove in particular that number of points of a tubular 
set $\X$ of degree $\delta$ and dimension $d$ equals
$$
|\X (\F_q)| = \delta q^d + \pi_{d-1}.
$$

We observed in \S \ref{ss:optimal} that, for general equidimensional
varieties, the upper bound on the number of rational points is reached
by arrangements of linear varieties.
That fact was already known for hypersurfaces
\cite{lettre}. If this property holds for complete intersections, one does
not know arrangements of linear varieties which are complete intersection
and have more points than tubular
sets. For this reason one can hope the existence of a sharper bound for 
the maximal number of points of complete intersection with respect to their
dimension and degree. This problem remains completely open.

\section*{Conclusion}

We obtained in Theorem \ref{thm:main} a new upper bound on the number of
rational points of an arbitrary closed subset of a projective space.
This bound holds even for non equidimensional varieties. In the
equidimensional case, thanks to this upper bound, we proved Ghorpade and
Lachaud's conjecture \cite{Ghorpade_Lachaud} is true and is optimal
for equidimensional varieties.

\section*{Acknowledgements}
The author expresses his gratitude to the anonymous referee for his
careful work and his relevant suggestions.

\bibliographystyle{abbrv}
\bibliography{biblio}

\begin{thebibliography}{10}

\bibitem{AP2}
Y.~Aubry and M.~Perret.
\newblock Coverings of singular curves over finite fields.
\newblock {\em Manuscripta Math.}, 88(1):467--478, 1995.

\bibitem{aubryperret}
Y.~Aubry and M.~Perret.
\newblock On the characteristic polynomials of the {F}robenius endomorphism for
  projective curves over finite fields.
\newblock {\em Finite Fields Appl.}, 10(3):412--431, 2004.

\bibitem{Beutelspacher}
A.~Beutelspacher.
\newblock Blocking sets and partial spreads in finite projective spaces.
\newblock {\em Geom. Dedicata}, 9(4):425--449, 1980.

\bibitem{Bog}
M.~Boguslavsky.
\newblock On the number of solutions of polynomial systems.
\newblock {\em Finite Fields Appl.}, 3(4):287--299, 1997.

\bibitem{DG2}
M.~Datta and S.~Ghorpade.
\newblock Number of solutions of systems of homogeneous polynomials over finite
  fields.
\newblock ArXiv:1507.03029, 2015.

\bibitem{DG}
M.~Datta and S.~Ghorpade.
\newblock On a conjecture of {T}sfasman and an inequality of {S}erre for the
  number of points on hypersurfaces over finite fields.
\newblock ArXiv:1503.03049, 2015.

\bibitem{deligne}
P.~Deligne.
\newblock La conjecture de {W}eil. {I}.
\newblock {\em Publ. Math. Inst. Hautes \'Etudes Sci.}, pages 273--307, 1974.

\bibitem{EH}
D.~Eisenbud and J.~Harris.
\newblock {\em The geometry of schemes}, volume 197 of {\em Graduate Texts in
  Mathematics}.
\newblock Springer-Verlag, New York, 2000.

\bibitem{Fulton_intersection}
W.~Fulton.
\newblock {\em Intersection theory}, volume~2 of {\em Ergebnisse der Mathematik
  und ihrer Grenzgebiete. 3. Folge. A Series of Modern Surveys in Mathematics}.
\newblock Springer-Verlag, Berlin, second edition, 1998.

\bibitem{Ghorpade_Lachaud}
S.~R. Ghorpade and G.~Lachaud.
\newblock {\'E}tale cohomology, {L}efschetz theorems and number of points of
  singular varieties over finite fields.
\newblock {\em Mosc. Math. J.}, 2(3):589--631, 2002.

\bibitem{H}
R.~Hartshorne.
\newblock {\em Algebraic geometry}, volume~52 of {\em Graduate Texts in
  Mathematics}.
\newblock Springer-Verlag, New York, 1977.

\bibitem{HommaKim}
M.~Homma and S.~J. Kim.
\newblock Sziklai's conjecture on the number of points of a plane curve over a
  finite field~{III}.
\newblock {\em Finite Fields Appl.}, 16(5):315 -- 319, 2010.

\bibitem{Homma_Kim_hyp}
M.~Homma and S.~J. Kim.
\newblock An elementary bound for the number of points of a hypersurface over a
  finite field.
\newblock {\em Finite Fields Appl.}, 20(0):76 -- 83, 2013.

\bibitem{Jungnickel1984}
D.~Jungnickel.
\newblock Maximal partial spreads and translation nets of small deficiency.
\newblock {\em J. Algebra}, 90(1):119 -- 132, 1984.

\bibitem{Jungnickel1993}
D.~Jungnickel.
\newblock Maximal partial spreads and transversal-free translation nets.
\newblock {\em J. Combin. Theory Ser. A}, 62(1):66 -- 92, 1993.

\bibitem{LR}
G.~Lachaud and R.~Rolland.
\newblock An overview of the number of points of algebraic sets over finite
  fields.
\newblock ArXiv:1405.3027v2, 2014.

\bibitem{Storme}
K.~Metsch and L.~Storme.
\newblock Partial t-spreads in {$PG(2t+1,q$)}.
\newblock {\em Des. Codes Cryptogr.}, 18(1-3):199--216, 1999.

\bibitem{Serre_borne}
J.-P. Serre.
\newblock Sur le nombre de points rationnels d'une courbe alg{\'e}brique sur un
  corps fini.
\newblock {\em C. R. Acad. Sci. Paris Ser. I}, 296:397--402, 1983.

\bibitem{lettre}
J.-P. Serre.
\newblock Lettre \`a {M}. {T}sfasman.
\newblock {\em Ast\'erisque}, 198-200:351--353, 1991.
\newblock Journ{\'e}es Arith\-m{\'e}tiques, 1989 (Luminy).

\bibitem{Sorensen_DiscreteMath}
A.~B. S{\o}rensen.
\newblock On the number of rational points on codimension-1 algebraic sets in
  $\mathbb{P}^n(\mathbb{F}_q)$.
\newblock {\em Discrete Math.}, 135(1--3):321--334, 1994.

\bibitem{StohrVoloch}
K.-O. St{\"o}hr and J.~F. Voloch.
\newblock Weierstrass points and curves over finite fields.
\newblock {\em Proc. London Math. Soc}, pages 1--19, 1986.

\bibitem{Sziklai}
P.~Sziklai.
\newblock A bound on the number of points of a plane curve.
\newblock {\em Finite Fields Appl.}, 14(1):41--43, 2008.

\bibitem{Thas}
K.~Thas.
\newblock On the number of points of a hypersurface in finite projective space
  (after {J}.-{P}. {S}erre).
\newblock {\em Ars Combin.}, 94:183--190, 2010.

\end{thebibliography}

\end{document}